\documentclass[11pt,reqno]{amsart}
\usepackage{amssymb,latexsym}
\usepackage{graphicx}
\usepackage{url}		%does nice formatting of URLs

\calclayout

\makeatletter
\newcommand{\monthyear}[1]{%
  \def\@monthyear{\uppercase{#1}}}
\newcommand{\volnumber}[1]{%
  \def\@volnumber{\uppercase{#1}}}
\AtBeginDocument{%
\def\ps@plain{\ps@empty
  \def\@oddfoot{\@monthyear \hfil \thepage}%
  \def\@evenfoot{\thepage \hfil \@volnumber}}
\def\ps@firstpage{\ps@plain}
\def\ps@headings{\ps@empty
  \def\@evenhead{%
    \setTrue{runhead}%
    \def\thanks{\protect\thanks@warning}%
    \uppercase{The Fibonacci Quarterly}\hfil}%
  \def\@oddhead{%
    \setTrue{runhead}%
    \def\thanks{\protect\thanks@warning}%
    \hfill\uppercase{Repdigits as products}}%
  \let\@mkboth\markboth
  \def\@evenfoot{%
    \thepage \hfil \@volnumber}%
  \def\@oddfoot{%
    \@monthyear \hfil \thepage}%
  }%
\footskip=25pt
\pagestyle{headings}%
}
\makeatother

\theoremstyle{plain}
\numberwithin{equation}{section}
\newtheorem{thm}{Theorem}[section]
\newtheorem{theorem}[thm]{Theorem}

\begin{document}
%% replace the values in the next three lines by the correct information
\monthyear{November 2018}
\volnumber{56, 4}
\setcounter{page}{319}

\title{Repdigits as products of consecutive balancing or Lucas-balancing numbers}
\author{Sai Gopal Rayaguru}
\address{Department of Mathematics\\
                National Institute of Technology\\
                Rourkela-769008, Odisha\\
                India}
\email{saigopalrs@gmail.com}
\thanks{}
\author{Gopal Krishna Panda}
\address{Department of Mathematics\\
               National Institute of Technology\\
               Rourkela-769008, Odisha\\
               India}
\email{gkpanda\_nit@rediffmail.com}

\begin{abstract}
Repdigits are natural numbers formed by the repetition of a single digit. In this paper, we explore the presence of repdigits in the product of consecutive balancing or Lucas-balancing numbers.\\

\end{abstract}

\maketitle

\section{Introduction}
The  balancing sequence $\{B_n:n\geq0\}$ and the Lucas-balancing sequence $\{C_n:n\geq0\}$ are solutions of the binary recurrence $x_{n+1} = 6x_{n} - x_{n-1}$ with initial terms $B_0 =0, B_1 = 1$ and $C_0 =1, C_1 = 3$  respectively. The balancing sequence is a variant of the sequence of natural numbers since natural numbers are solutions of the binary recurrence $x_{n + 1} = 2x_{n} - x_{n-1}$ with initial terms $x_0 =0, x_1 = 1$. The balancing numbers have certain properties identical with those of natural numbers \cite{Ray}. It is important to note that the balancing sequence is a strong divisibility sequence, that is, $B_m \mid B_n$ if and only if $m \mid n$ \cite{Panda2009}.\\

In the year 2004, Liptai \cite{Liptai} searched for Fibonacci numbers in the balancing sequence and observed that $1$ is the only number of this type. In a recent paper \cite{Panda2012}, the second author proved that there is no perfect square in the balancing sequence other than $1$. Subsequently, Panda and Davala \cite{Panda2015} verified that $6$ is the only balancing number which is also a perfect number.\\ 

For a given integer $g>1,$ a number of the form $\textit{N}=a\big(\frac{g^m-1}{g-1}\big)$ for some $m\geq 1$ where $a\in\{1,2,\cdots,g-1\}$ is called a repdigit with respect to base $g$ or $g$-repdigit. For $g=10$, $N$ is simply called a repdigit and if, in addition, $a=1$, then $N$ is called a repunit. Luca \cite{Luca} identified the repdigits in Fibonacci and Lucas sequences. Subsequently, Faye and Luca \cite{Faye} explored all repdigits in Pell and Pell-Lucas sequences. Marques and Togb\'{e} \cite{Marques} searched for the repdigits which are product of consecutive Fibonacci numbers. In this paper, we search for repdigits in the balancing and Lucas-balancing sequences. In addition, we also explore repdigits which are product of consecutive balancing or Lucas-balancing numbers.

\section{Main Results}

In this section, we prove some theorems assuring the absence of certain class  of repdigits in the balancing and Lucas-balancing sequences. As generalizations, we also show that the product of consecutive balancing or Lucas-balancing numbers is never a repdigit with more than one digit.\\

In the balancing sequence, the first two balancing numbers $B_1=1$ and $B_2=6$ are repdigits. We have checked the next 200 balancing numbers, but none is a repdigit. The following theorem excludes the presence of some specific types of repdigits in the balancing sequence.

%Theorem 2.1

\begin{theorem}\label{A}
If $m,n$ and $a$ are natural numbers, $m\geq2$, $a\neq6$, and $1\leq a\leq9$, then the Diophantine equation
\begin{equation}\label{1}
B_n=a \Big( \frac{10^m-1}{9}\Big)
\end{equation}
 has no solution. 
\end{theorem}
\begin{proof}
 To prove this theorem, we need all the least residues of the balancing sequence modulo $3,4,5,7,8,9,11$ and $20$ (see \cite{Panda2014}). We list them in the following table.\\

\begin{center}
\begin{table}[ht]
\begin{tabular}{|c|c|c|c|}
\hline
Row no. & $m$ & $B_n$ mod $m$ & Period \\
[0.5ex]
\hline
$1$ & $3$ & $0, 1, 0, 2$ & $4$ \\ 
\hline
$2$ & $4$ & $0,1,2,3$ & $4$ \\
\hline
$3$ & $5$ & $0,1,1,0,4,4$ & $6$ \\
\hline
$4$ & $7$ & $0,1,6$ & $3$ \\
\hline
$5$ & $8$ & $0,1,6,3,4,5,2,7$ & $8$ \\
\hline
$6$ & $9$ & $0,1,6,8,6,1,0,8,3,1,3,8$ & $12$ \\
\hline
$7$ & $11$ & $0,1,6,2,6,1,0,10,5,9,5,10$ & $12$ \\
\hline
$8$ & $20$ & $0,1,6,15,4,9,10,11,16,5,14,19$ & $12$ \\
\hline
\end{tabular}
\vspace{0.2cm}
\caption{}
\label{tab:1}
\end{table}
\end{center}

Since $m\geq2$, it follows that $n\geq3$. We claim that $m$ is odd. Observe that if $m$ is even, then $$11\mid \frac{10^m-1}{9}\mid B_n$$ and from the seventh row of Table~\ref{tab:1}, it follows that $6\mid n$ and consequently $B_6\mid B_n$. Since $10\mid B_6$, it follows that $10\mid B_n=a \cdot \frac{10^m-1}{9}$, which is a contradiction. Now, to complete the proof, we distinguish eight different cases corresponding to the values of $a$.
\vspace{0.3 cm}\\
\textbf{Case I: $a=1$}. Assume that $B_n$ is of the form $\frac{10^m-1}{9}$ for some $m$. Since $m$ is odd, $B_n \equiv 1~ (\text{mod}~ 11)$ and also $B_n \equiv 11~ (\text{mod}~ 20)$. From the last row of Table~\ref{tab:1}, it follows that if $B_n \equiv 11~ (\text{mod}~ 20)$ then $n\equiv 7~ (\text{mod}~ 12)$. But, from the seventh row of Table~\ref{tab:1}, it follows that whenever $n\equiv 7~ (\text{mod}~ 12)$, $B_n\equiv 10~ (\text{mod}~ 11)-$a contradiction to  $B_n\equiv 1~ (\text{mod}~ 11)$. Hence, no $B_n$ is of the form $\frac{10^m-1}{9}$.
\vspace{0.3 cm}\\
\textbf{Case II: $a=2$}. If $B_n=2 \cdot \frac{10^m-1}{9}$, then $B_n \equiv 2~ (\text{mod}~ 5)$. But, in view of the third row of Table~\ref{tab:1}, it follows that for no value of $n$,  $B_n\equiv 2~(\text{mod}~ 5).$ Hence, $B_n$ cannot be of the form $2 \cdot \frac{10^m-1}{9}.$
\vspace{0.3 cm}\\
\textbf{Case III: $a=3$}. If $B_n=3 \cdot \frac{10^m-1}{9}$, then $B_n \equiv 0~(\text{mod}~ 3)$. But, in view of the first row of Table~\ref{tab:1}, $n \equiv 0,2~ (\text{mod}~ 4)$. So, $B_2 \mid B_n$ and consequently $2 \mid \frac{10^m-1}{9}$, which is a contradiction. Hence, $B_n$ cannot be of the form $3 \cdot \frac{10^m-1}{9}.$
\vspace{0.3 cm}\\
\textbf{Case IV: $a=4$}. If $B_n=4 \cdot \frac{10^m-1}{9}$, then $B_n \equiv 0 ~(\text{mod}~ 4)$ and in view of the second row of Table~\ref{tab:1}, $ 4 \mid n$ which implies $B_4 \mid B_n.$ Since $17 \mid B_4,$ it follows that $17 \mid (10^m-1).$ But this is possible if $16 \mid m,$  which is a contradiction since $m$ is odd. Hence, $B_n$ cannot be of the form $B_n=4 \cdot \frac{10^m-1}{9}.$
\vspace{0.3 cm}\\
\textbf{Case V: $a=5$}. If $B_n=5 \cdot \frac{10^m-1}{9}$, then $B_n \equiv 0 ~(\text{mod}~ 5)$ and in view of the third row of Table~\ref{tab:1}, this is possible only if $3 \mid n.$ Hence, $B_3 \mid B_n$ and since $7 \mid B_3,$ it follows that $7 \mid \frac{10^m-1}{9}$ which implies that $ 6 \mid m,$ a contradiction since $m$ is odd. Hence, $B_n$ cannot be of the form $B_n=5 \cdot \frac{10^m-1}{9}.$
\vspace{0.3 cm}\\
\textbf{Case VI: $a=7$}. If $B_n=7 \cdot \frac{10^m-1}{9}$, then $B_n \equiv 0 ~(\text{mod}~ 7)$ and in view of the fourth row of Table~\ref{tab:1}, this is possible only if $3 \mid n.$ Hence, $B_3 \mid B_n$ and since $5 \mid B_3,$ it follows that $5 \mid \frac{10^m-1}{9},$ which is a contradiction. Hence, $B_n$ cannot be of the form $B_n=7 \cdot \frac{10^m-1}{9}$.
\vspace{0.3 cm}\\
\textbf{Case VII: $a=8$}. If $B_n=8 \cdot \frac{10^m-1}{9}$, then $B_n \equiv 0 ~(\text{mod}~ 8)$ and in view of the fifth row of Table~\ref{tab:1}, this is possible only if $8 \mid n.$ Hence, $B_8 \mid B_n$ and since $17 \mid B_8,$ it follows that $17 \mid (10^m-1).$ But this is possible if $16 \mid m,$  which is a contradiction since $m$ is odd. Hence, $B_n$ cannot be of the form $B_n=8 \cdot \frac{10^m-1}{9}$.
\vspace{0.3 cm}\\
\textbf{Case VIII: $a=9$}. If $B_n=9 \cdot \frac{10^m-1}{9}$, then $B_n \equiv 0 ~(\text{mod}~ 9)$ and in view of the sixth row of Table~\ref{tab:1}, this is possible only if $6 \mid n.$ Consequently, $B_6 \mid B_n$ and since $11 \mid B_6,$ it follows that $11 \mid \frac{10^m-1}{9}$. But this is possible only if $m$ is even, which is a  contradiction since $m$ is odd. Hence, $B_n$ cannot be of the form $B_n=9 \cdot \frac{10^m-1}{9}$.\\

Thus, \eqref{1} has no solution if $m\geq2$ and $a\neq6$. This completes the proof.
\end{proof}

We next study the presence of repdigits in the products of consecutive balancing numbers. The product $B_1B_2=6$ is a repdigit. So a natural question is: "Is there any other repdigit which is a consecutive product of balancing numbers?" In the following theorem, we answer this question in negative.\\

%Theorem 2.2
 
\begin{theorem}\label{B}
If $m,n,k$ and $a$ are natural numbers such that $m>1$ and $1\leq a\leq9$, then the Diophantine equation
\begin{equation}\label{2}
B_nB_{n+1}\cdots B_{n+k}=a \Big( \frac{10^m-1}{9} \Big)
\end{equation}
has no solution.
\end{theorem}
\begin{proof}

Firstly, we show that \eqref{2} has no solution for $k \geq 2$. Assume to the contrary that \eqref{2} has a solution in positive integers $n,m,a$ for $k\geq2$. Then, $2 \mid (n+i)$ and $3 \mid (n+j)$ for some $i,j \in\{0,1,\ldots,k\}$. Since $2 \mid B_2$ and $5 \mid B_3$, it follows that $2 \mid B_{n+i}$ and $5 \mid B_{n+j}$. Hence, $10 \mid B_nB_{n+1}\cdots B_{n+k}=a \Big( \frac{10^m-1}{9} \Big)$, which is a contradiction. Hence, \eqref{2} has no solution for $k \geq 2$. \\

We next show that \eqref{2} has no solution if $k=1$. If $k=1$, \eqref{2} reduces to 
$$B_nB_{n+1} = a \Big( \frac{10^m-1}{9} \Big).$$
One of $n$ and $n+1$ is even and consequently, either $B_n$ or $B_{n+1}$ is also even. Hence, $a \in\{2,4,6,8\}$. Since $m>1$, $B_nB_{n+1}\geq11$ and hence $n$ must be greater than $1$.\\
 
In the following table we list all the least residues of $B_nB_{n+1}$ modulo $5$ and $100$, which will be useful in the proof.

\begin{center}
\begin{table}[ht]
\begin{tabular}{|c|c|c|} 
\hline
$m$ & $B_nB_{n+1}$ mod $m$ & Period \\
[0.5ex]
\hline
$5$ & $0,1,0$ & $3$ \\ 
\hline
 & $0,6,10,40,56,70,30,56,80,70,6,40,60,6,50,0,56,10,90,56,$ &  \\
$100$ & $20,30,6,80,20,6,90,60,56,50,50,56,60,90,6,20,80,6,30,20,$ & $60$ \\
  & $56,90,10,56,0,50,6,60,40,6,70,80,56,30,70,56,40,10,6,0$ &  \\ [0.5ex] 
\hline
\end{tabular}
\vspace{0.2cm}
\caption{}
\label{tab:2}
\end{table}
\end{center}

If $a=2$ or $a=4$, then $$B_nB_{n+1}=a \cdot\frac{10^m-1}{9}\equiv a ~(\text{mod}~ 5).$$\\
If $a=8$, then
$$B_nB_{n+1}=8 \cdot\frac{10^m-1}{9}\equiv 3 ~(\text{mod}~ 5).$$\\
Similarly, if $a=6$, then 
$$B_nB_{n+1} = 6 \cdot\frac{10^m-1}{9} \equiv 66 ~(\text{mod}~ 100).$$
Since the least residues of the last three congruences do not appear in the appropriate row of Table~\ref{tab:2}, it follows that $B_nB_{n+1}$ is not a repdigit if $n>1$. This completes the proof.  
\end{proof}
\vspace{0.2 cm}

In Theorem \ref{A}, we proved the absence of certain type of repdigits in the sequence of balancing numbers. However, in case of Lucas-balancing numbers, $C_1=3$ and $C_3=99$ are two known repdigits. Thus, a natural question is: "Does this sequence contain any other larger repdigit?" In the following theorem, we answer this question in negative.\\

%Theorem 2.3

\begin{theorem}\label{C}
If $m,n$ and $a$ are natural numbers and $1\leq a\leq9$, then the Diophantine equation
\begin{equation}\label{3}
C_n=a \Big( \frac{10^m-1}{9}\Big)
\end{equation}
 has the only solutions $(m,n,a)=(1,1,3),(2,3,9)$. 
\end{theorem}
\begin{proof}
To prove this theorem, we need all the least residues of the Lucas-balancing sequence modulo $5,7$ and $8$. We list them in the following table.\\

\begin{center}
\begin{table}[ht]
\begin{tabular}{|c|c|c|c|} 
\hline
Row no. & $m$ & $C_n$ mod $m$ & Period \\
[0.5ex]
\hline
$1$ & $5$ & $1,3,2,4,2,3$ & $6$ \\
\hline
$2$ & $7$ & $1,3,3$ & $3$ \\
\hline
$3$ & $8$ & $1,3$ & $2$ \\
\hline
\end{tabular}
\vspace{0.2cm}
\caption{}
\label{tab:3}
\end{table}
\end{center}

Among the first three Lucas-balancing numbers $C_1=3$ and $C_3=99$ are repdigits and \eqref{3} is satisfied for $(m,n,a)=(1,1,3),(2,3,9)$. Now, let  $n\geq 4$ and hence $m\geq 3$.  Since $C_n$ is always odd, $a \in \{ 1,3,5,7,9\}.$ Since no zero appears in the first two rows of Table~\ref{tab:3}, it follows that $C_n$ is not divisible by $5$ or $7$ and hence the possible values of $a$ are limited to $1,3,9$.\\

If $a\in \{1,9\}$, then
$$C_n=a \cdot\frac{10^m-1}{9}\equiv 10^m-1 \equiv 7 ~(\text{mod}~ 8).$$ Similarly, if $a=3$, then
$$C_n=3 \cdot \frac{10^m-1}{9}\equiv 5 ~(\text{mod}~ 8).$$
Since, the least residues $5$ and $7$ do not appear in the last row of Table~\ref{tab:3}, it follows that \eqref{3} has no solution for $n>3$. This completes the proof.
\end{proof}

In Theorem \ref{B}, we noticed that no product of consecutive balancing numbers is a repdidit with more than one digit, though the only product $B_1B_2=6$ is a single digit repdigit. The following theorem negates the possibility of any repdigit as product of consecutive Lucas-balancing numbers.\\ 

%Theorem 2.4

\begin{theorem}\label{D}
If $m,n,k$ and $a$ are natural numbers and $1\leq a\leq9$, then the Diophantine equation
\begin{equation}\label{4}
C_nC_{n+1}\cdots C_{n+k}=a \Big( \frac{10^m-1}{9} \Big)
\end{equation}
has no solution.
\end{theorem}
\begin{proof}
All the Lucas-balancing numbers are odd and in view of \eqref{4}, $a \in\{1,3,5,7,9\}$. It is easy to see that \eqref{4} has no solution if $m=1,2$. In the following table we list all the nonnegative residues of Lucas-balancing numbers and their consecutive product  modulo $5,7$ and $8$ which will play an important role in proving this theorem.\\

\begin{center}
\begin{table}[ht]
\begin{tabular}{|c|c|c|} 
\hline
$m$ & $C_n$ mod $m$ &$C_nC_{n+1}\cdots C_{n+k}$ mod $m$ \\
[0.5ex]
\hline
$5$ & $1,3,2,4,2,3$ & $\in\{1,2,3,4\}$ \\ 
\hline
$7$ & $1,3,3$ & $\in\{1,2,3,4,5,6\}$ \\
\hline
$8$ & $1,3$ & $\in\{1,3\}$ \\[0.5ex] 
\hline
\end{tabular}
\vspace{0.2cm}
\caption{}
\label{tab:4}
\end{table}
\end{center}

For $m\geq 3$, $C_nC_{n+1}\cdots C_{n+k}=a \Big( \frac{10^m-1}{9} \Big) \equiv 7a~(\text{mod}~ 8)$. But from the last row of Table~\ref{tab:4}, it follows that $7a \equiv 1,3~(\text{mod}~ 8)$ and hence $a=5$ or $a=7$. Now, reducing \eqref{4} modulo $a$ we get $ C_nC_{n+1}\cdots C_{n+k} \equiv 0~(\text{mod}~ a)$. Since, $0$ does not appear as a residue of $C_nC_{n+1}\cdots C_{n+k}$ modulo $5$ or $7$, it follows that \eqref{4} has no solution for $m\geq3$. This completes the proof.\\
\end{proof}

\section{Conclusion}
In the last section, we noticed that the Lucas-balancing sequence contains only two repdigits namely $C_1=3$ and $C_3=99$, while we could not explore all repdigits in the balancing sequence. In Theorem \ref{A}, we proved that $B_n$ is not a repdigit ($B_n\neq a \Big(\frac{10^m-1}{9}\Big)$), with more than one digit, if $a\neq6$. Thus, repdigits in the balancing sequence having all digits $6$ is yet unexplored. In this connection, one can verify that if $n\not\equiv 14 ({\text{mod}}~96)$ then $B_n$ is not a repdigit. Further, if $m\not\equiv 1({\text{mod}}~6)$, then also $B_n$ is not a repdigit. We believe that, $B_1=1$ and $B_2=6$ are the only repdigits in the balancing sequence. It is still an open problem to prove the existence or nonexistence of repdigits that are 6 times of some repunit other than $B_2=6$.
\vspace{0.5cm}\\    
\textbf{Acknowledgment:} It is a pleasure to thank the anonymous referee whose comments helped us in improving the paper to a great extent.\\

\medskip

\noindent MSC2010: 11B39, 11A63, 11B50.

\end{document}